\documentclass[11pt, reqno]{amsart}
\usepackage{setspace}
\usepackage{amsmath, amssymb, amscd, amsthm, amsfonts}
\usepackage{graphicx}
\usepackage{hyperref}
\usepackage{cleveref}
\usepackage{dirtytalk}

\usepackage{mathpazo}
\usepackage{indentfirst}
\usepackage{amssymb,amsmath,amsfonts,amsthm}
\usepackage[all]{xy}
\usepackage{enumerate}

\usepackage{graphicx}
\numberwithin{equation}{section}
\usepackage[all]{xy}
\usepackage{amssymb,amscd,amsthm,amsmath,graphicx,color}
\usepackage{mathrsfs}
\usepackage[english]{babel}
\usepackage[utf8x]{inputenc}
\usepackage{hyperref}

\newtheorem{theorem}{Theorem}[section]
\newtheorem*{proposition*}{Proposition}
\newtheorem{proposition}[theorem]{Proposition}
\newtheorem{corollary}[theorem]{Corollary}
\newtheorem{lemma}[theorem]{Lemma}
\newtheorem{example}[theorem]{Example}
\newtheorem{remark}[theorem]{Remark}
\newtheorem{definition}[theorem]{Definition}
\newtheorem{question}[theorem]{Question}

\DeclareMathOperator{\Hom}{Hom}
\DeclareMathOperator{\Ext}{Ext}
\DeclareMathOperator{\Tor}{Tor}

\DeclareMathOperator{\depth}{depth}

\DeclareMathOperator{\pd}{pd}

\DeclareMathOperator{\gdim}{G-dim}
\DeclareMathOperator{\id}{id}
\DeclareMathOperator{\gid}{Gid}

\begin{document} 

\title[Finite projective dimension]{Finite projective dimension and a question of Jorgensen}

\author[R. Holanda]{Rafael Holanda}
\address{Departamento de Matem\'atica, CCEN, Universidade Federal de Pernambuco, Recife, PE, 50740-560, Brazil}
\email{rafael.holanda@ufpe.br}

\author{Cleto B. Miranda-Neto}
\address{Departamento de Matemática, Universidade Federal da Paraíba - 58051-900, João Pessoa, PB, Brazil}
\email{cleto@mat.ufpb.br}

\date{\today}

\keywords{Finite projective dimension, homological dimension, Ext module, free module}
\subjclass[2020]{Primary 13D05, 13D07, 13H10, 13C10; Secondary 13C13, 13D02, 13D45.}

\begin{abstract} This paper studies finite projective dimension of finitely generated modules over a Noetherian local ring, by means of spectral sequence methods related to generalized local cohomology. Our main goal is to  address a question raised by D. Jorgensen over fifteen years ago, concerning a prescribed bound (via Ext vanishing) for projective dimension over a complete intersection local ring. We obtain similar results involving other homological dimensions as well. Also we make use of weakly full ideals to derive further criteria for prescribed bound on projective dimension.

\end{abstract}

\maketitle 

\section{Introduction}

A Noetherian local ring $(R, \mathfrak{m})$ is said to be a complete intersection if its  $\mathfrak{m}$-adic completion  is isomorphic to the quotient of a regular local ring by a (possibly zero) proper ideal generated by a regular sequence. Our motivation in this note is the following intriguing problem suggested by Jorgensen in 2008.

\begin{question}\label{Jquestion}{\rm (\cite[Question 1.7]{J})}\label{jquestion} Let $R$ be a complete intersection local ring of positive codimension, and let $M$ be a finitely generated $R$-module with finite projective dimension. Assume $\Ext^n_R(M,M)=0$ for some integer $n\geq 1$. Is it true that the projective dimension of $M$ is at most $n-1$?
\end{question}

Notice that we can assume $n\leq d={\rm dim}\,R$ (since $M$ has finite projective dimension). As far as we know, Jorgensen's problem remains unsettled in its generality and has only been confirmed in a very few special cases, as listed below ($R$, $M$  and $n$ are as in Question \ref{Jquestion}).

\begin{enumerate}

\item[{\rm (a)}] (\cite[Proposition 2.5]{J}) The case $n=2$.

\item[{\rm (b)}] (\cite[Proposition 5.4]{Dao})
$[M]=0$ in ${\overline G}(R)\otimes_{\mathbb Z}{\mathbb Q}$, where ${\overline G}(R)=G(R)/{\mathbb Z}\cdot [R]$ denotes the reduced Grothendieck group of $R$.

\item[{\rm (c)}] (\cite[Corollary 2.10]{Soto})
$R$ is an unramified hypersurface and $M$ is Cohen-Macaulay.

\end{enumerate}

In the present work, our main goal is to prove the following contribution towards a positive solution to Jorgensen's question, over a broader class of rings.

\begin{theorem} {\rm (Corollary \ref{janswer})} Assume that $R$ is a $($not necessarily complete intersection$)$ Gorenstein local ring. Then, Question \ref{jquestion} has an affirmative answer if, in addition,  $$\depth_R\Ext^q_R(M,M)\geq d-n-q \quad \mbox{for\, all} \quad q=0, \ldots, d-n.$$ 
\end{theorem}

As will be clear, this fact follows readily from a more general result (Theorem
\ref{prescribedpd}), established for the case where $R$ is Cohen-Macaulay and admits a canonical module. 

Note that, as ${\rm depth}_R0=+\infty$, the above depth condition is automatically satisfied if, for example, ${\rm depth}_R{\rm Hom}_R(M, M)\geq d-n$ and (for $n<d$, as we may clearly assume) $$\Ext^q_R(M,M)=0\quad \mbox{for\, all} \quad q=1, \ldots, d-n.$$

In Section \ref{otherhd}, we study connections between the vanishing of Ext modules and the finiteness of the injective dimension and the Gorenstein injective dimension of a module over a given Cohen-Macaulay local ring. The main result is Theorem \ref{prescribedpdhom} (see also Corollary \ref{prescribedpdhom-cor}). Another byproduct, which put us again in the scenario of finite projective dimension, is Corollary \ref{cor-e}.

Our interest in Section \ref{sectionfinitepd} meets some of the preceding ones, but via different methods. We use a special type of strongly rigid module, namely, ideals that are weakly full with respect to a suitable power of the maximal ideal (e.g., integrally closed ideals, provided that the ring has positive depth) to establish the prescribed bound on projective dimension given in Proposition
\ref{weakly-strongly}. 

\smallskip

For any unexplained terminology used in this note, we refer to \cite{BH}. As a matter of convention (which will be in force throughout the entire paper), by the expression \textit{ring} we  always mean commutative unitary Noetherian ring. We say that a module (over a given ring) is \textit{finite} whenever it is finitely generated. Over a local ring $(R, {\mathfrak m})$, by depth we always mean ${\mathfrak m}$-depth. As mentioned above, we set the depth of the zero module as $+\infty$.

\section{Finite projective dimension}\label{hdsection}

The projective dimension $\pd_RM$ of a finite module $M$ over a local ring $R$ is a fundamental classical invariant. Our main goal here is to address Jorgensen's problem (Question \ref{jquestion}) concerning a potential cohomological criterion for a prescribed upper bound on $\pd_RM$ to hold, based on a single Ext vanishing. Before tackling such problem, some preparation is in order.

The first auxiliary notion is that of generalized local cohomology, which plays an important role in this paper.

\begin{definition}\rm (\cite{H}) 
Let $R$ be a ring and $M,N$ be finite $R$-modules. Given an ideal $\mathfrak{a}$ of $R$ and an integer $i\geq 0$, the \emph{$i$th generalized local cohomology module of the pair $M$, $N$ with respect to $\mathfrak{a}$} is defined as $H_\mathfrak{a}^i(M,N)=\varinjlim\Ext^i_R(M/\mathfrak{a}^nM,N)$.
\end{definition}

Notice that by taking $M = R$ we retrieve the ordinary local cohomology module $H_\mathfrak{a}^i(N)$.

If $R$ is a local ring and  $M, N$ a pair of finite $R$-modules, we consider the quantity $$e_R(M,N)=\sup\{j\geq 0 \, \mid \, \Ext^j_R(M,N)\neq0\}.$$

\begin{lemma}{\rm (\cite[Lemma 1 (iii)]{M}})\label{pdfinite}
Let $R$ be a local ring and let $M, N$ be finite $R$-modules. If $\pd_RM<\infty$ and $N\neq0$, then $e_R(M,N)=\pd_RM$, that is, $e_R(M,N)=\depth R-\depth_RM$.
\end{lemma}

\begin{lemma}{\rm (\cite[Proposition 2.1]{FJMS})}\label{ss2}
Let $(R, \mathfrak{m})$ be a local ring. If $M, N$ are finite $R$-modules, there exists a spectral sequence $H^p_\mathfrak{m}(\Ext_R^q(M,N))\Rightarrow_p H^{p+q}_\mathfrak{m}(M,N)$.
    
\end{lemma}

\begin{lemma}{\rm (\cite[Lemma 2.2]{Y})}\label{tor0}
Let $R$ be a local ring and $M, N$ be finite $R$-modules with $\pd_RM<\infty$ and $N$ maximal Cohen-Macaulay. Then, $\Tor_j^R(M,N)=0$ for all $j\geq1$.
\end{lemma}

\begin{lemma}{\rm (\cite[Theorem 2.3]{S})}\label{depthlemma}
Let $(R, \mathfrak{m})$ be a local ring and $M, N$ be finite $R$-modules. Set $t=\depth_RN$. Then, $H^t_\mathfrak{m}(M,N)\neq0$ and $H^j_\mathfrak{m}(M,N)=0$ for all $j<t$.
\end{lemma}

In the next lemma, and wherever it appears throughout the paper, the notation $H^{\vee}$ will stand for the Matlis dual of a module $H$ over a given local ring $(R, \mathfrak{m})$. In addition, the $\mathfrak{m}$-adic completion of a finite $R$-module $M$ will be denoted $\widehat{M}$\, ($\cong M\otimes_R\widehat{R}$).

\begin{lemma}{\rm (\cite[Theorem 2.1(a)]{HZ})}\label{duality}
Let $(R, \mathfrak{m})$ be a Cohen-Macaulay local ring of dimension $d$, and let $M, N$ be a pair of finite $R$-modules. If $\pd_RM<\infty$ then, for each $j\geq0$, there is an isomorphism $H^j_\mathfrak{m}(M,N)^\vee\cong\Ext^{d-j}_{\widehat{R}}(\widehat{N}, \widehat{M}\otimes_{\widehat{R}}\omega_{\widehat{R}})$.
\end{lemma}

We now introduce a couple of convenient conditions.

\begin{definition}\rm (The conditions $({\bf V}_{n, s})$ and $({\bf D}_{s})$) Let $R$ be a local ring and $M, N$ finite $R$-modules. 
Suppose $R$ has a canonical module $\omega_R$, and let us denote the module of (first-order) minimal syzygies of $\omega_R$ by  $\Omega^1\omega_R$ or simply by $\Omega\omega_R$. For positive integers $n\leq s\leq d={\rm dim}\,R$, we say that the pair $M, N$ satisfies the $({\bf V}_{n, s})$ condition if $$\Ext^j_R(M,N)=\Ext^{j+1}_R(M,N\otimes_R\Omega\omega_R)=0\quad \mbox{for \,all} \quad j=n, \ldots, s.$$ 


Next, even in case the local ring $R$ admits no canonical module, we can consider the following depth condition. A pair of finite $R$-modules $M, N$ has the $({\bf D}_{s})$ condition if $$\depth_R\Ext^q_R(M, N)\geq d-s-q\quad \mbox{for \,all} \quad q=0, \ldots, d-s.$$ 

In case the pair $M, M$ satisfies any of these properties, we will
refer to such condition as being a property of $M$ itself.

    
\end{definition}

Here is our main result in this paper.

\begin{theorem}\label{prescribedpd}
Let $R$ be a Cohen-Macaulay local ring of dimension $d$ with canonical module $\omega_R$. If $M$ is a finite $R$-module with $\pd_RM<\infty$ such that the $({\bf V}_{n, s})$ and $({\bf D}_{s})$ conditions hold for positive integers $n, s$ with $n\leq s\leq d$, then $\pd_RM<n$.
\end{theorem}
\begin{proof}
By Lemma \ref{pdfinite} we have $\pd_RM=e_R(M,M)$, so $e_R(M,M)\leq d$. If we first consider the case $d\leq s$, then we must have $e_R(M,M)<n$. Therefore, we may suppose $n\leq s<d$.

Since $\pd_RM<\infty$, Lemma \ref{tor0} forces $\Tor^R_1(M,\omega_R)=0$ and so there is a short exact sequence $$\xymatrix@=1em{0\ar[r] & M\otimes_R\Omega\omega_R\ar[r] & M\otimes_R F\ar[r] & M\otimes_R\omega_R\ar[r] & 0}$$ for some finite free $R$-module $F$. Hence, for each $j\geq0$, we get an exact sequence $$\xymatrix@=1em{\Ext^j_R(M,M\otimes_RF)\ar[r] & \Ext^j_R(M,M\otimes_R\omega_R)\ar[r] & \Ext^{j+1}_R(M,M\otimes_R\Omega\omega_R).}$$ By the $({\bf V}_{n, s})$ hypothesis, it follows that $\Ext^j_R(M,M\otimes_R\omega_R)=0$ for all $j=n, \ldots, s.$ Now, it should be noticed that taking $\mathfrak{m}$-adic completion (where $\mathfrak{m}$ is the maximal ideal of $R$) does not affect the conditions present in the statement, i.e., we may suppose that $R$ is complete. Thus, as $\pd_RM<\infty$, there are isomorphisms $$H^j_\mathfrak{m}(M,M)\cong\Ext^{d-j}_R(M,M\otimes_R\omega_R)^\vee=0 \quad \mbox{for \,all} \quad j=d-s, \ldots, d-n$$ by generalized local duality (see Lemma \ref{duality}). On the other hand, considering the spectral sequence given in Lemma \ref{ss2}, $$E_2^{p,q}=H^p_\mathfrak{m}(\Ext^q_R(M,M))\Rightarrow_p H^{p+q}_\mathfrak{m}(M,M),$$ the $({\bf D}_{s})$ hypothesis implies that $E_2^{p,q}=0$ for all $p<d-s-q$, that is, $p+q<d-s$. By convergence we conclude that $H^j_\mathfrak{m}(M,M)=0$ for all $j<d-s$. Therefore, $H^j_\mathfrak{m}(M,M)=0$ for all $j\leq d-n$ and, by Lemma \ref{depthlemma}, $\depth_RM>d-n$, so that $\pd_RM=d-\depth_RM<n$. \end{proof}

\begin{remark}\rm
We point out that, in the proof of Theorem \ref{prescribedpd}, the vanishing $\Tor_1^R(M,\omega_R)=0$ is assured if we ask $\gdim_RM<\infty$ instead of $\pd_RM<\infty$, see \cite[(3.4.6)]{Cr00}. However, in this situation, we are unable to apply local duality \ref{duality}. A similar obstacle arises in Theorem \ref{prescribedpdhom}.
\end{remark}

In the sequel, we will derive some immediate consequences of Theorem \ref{prescribedpd} by taking particular values of $n$ or $s$. First, we consider the case $n=1$, i.e., a characterization of freeness.

\begin{corollary}
Let $R$ be a Cohen-Macaulay local ring of dimension $d$ with canonical module $\omega_R$. If $M$ is a finite $R$-module with $\pd_RM<\infty$ such that the $({\bf V}_{1, s})$ and $({\bf D}_{s})$ conditions hold for some positive integer $s\leq d$, then $M$ is free.
\end{corollary}

Now we focus on Question \ref{jquestion}, which is the main problem we want to tackle in this paper.
Our Theorem \ref{prescribedpd} detects an additional (depth) condition under which Question \ref{jquestion} admits a positive answer. On the other hand, we do not require the ring to be a complete intersection. To be more precise, we consider the case where $R$ is Gorenstein and apply Theorem \ref{prescribedpd} with $s=n$ in order to record the following result, which gives a partial answer to Jorgensen's question (for the broader class of Gorenstein local rings). Note that the Gorenstein assumption implies the $({\bf V}_{n,n})$ condition.

\begin{corollary}\label{janswer}
Let $R$ be a Gorenstein local ring of dimension $d$. If $M$ is a finite $R$-module with $\pd_RM<\infty$ and $\Ext^n_R(M,M)=0$, such that the $({\bf D}_n)$ condition holds for some positive integer $n\leq d$, then $\pd_RM<n$.
\end{corollary}

Over complete intersections, we obtain a particularly interesting result. 

\begin{corollary}
Let $R$ be a complete intersection local ring of dimension $d$. If $M$ is a finite $R$-module with $\Ext^n_R(M,M)=0$, such that the $({\bf D}_n)$ condition holds for some positive even integer $n\leq d$, then $\pd_RM<n$.
\end{corollary}
\begin{proof} Since $R$ is a complete intersection and $n$ is even, it follows from \cite[Theorem 4.2]{AB} that $\pd_RM<\infty$. Now we apply Corollary \ref{janswer}.
\end{proof}

\medskip

We close the section illustrating the $({\bf D}_n)$ condition.

\begin{example}\label{exs}\rm (a) If as above $d$ denotes the dimension of the local ring $R$ then the case $n=d$ is trivial, i.e., the  $({\bf D}_d)$ condition is immediately seen to be satisfied by any pair of modules. On the other extreme, $({\bf D}_1)$ (i.e., the case $s=n=1$) is the strongest condition.

\medskip

(b) Let $k$ be any field, $m\geq 3$ and $R=k[\![x, y, z_3, \ldots, z_m]\!]/(xy)$. Then $d=m-1$. Clearly,  $0:_RyR=xR$. Now let $M=R/xR$ (notice, for completeness, that ${\rm pd}_RM=\infty$). Then $xM=0$ and, for all $i\geq 1$, we have ${\rm Ext}^{2i-1}_R(M, M)=0:_{M}yR=0$, and ${\rm Ext}^{2i}_R(M, M)=R/(x, y)R$.
Hence, ${\rm depth}_R{\rm Ext}_R^q(M, M)$ is $+\infty$ if $q$ is odd, and $m-2$ if $q$ is even. In addition, ${\rm Hom}_R(M, M)\cong M$, which has depth $m-1$. It follows that $M$ satisfies $({\bf D}_n)$ for all $n\leq m-1$.

\medskip

(c) In the previous example, $M$ was obtained from $R$ by modding out a zero-divisor. Here, let $(R, \mathfrak{m})$ be a Cohen-Macaulay local ring and pick the ideal $J$ generated by  an $R$-sequence $\{x_1, \ldots, x_{d-n}\}\subset {\mathfrak m}$, where $n$ is any integer satisfying $n< d\leq 2n$. Set $M=R/J$. We can write $$\Ext^q_R(M, M)\cong M^{\binom{d-n}{q}}\quad \mbox{for\, all} \quad q\geq 0,$$ which follows easily from the fact that the Koszul complex resolves regular sequences and the very definition of Ext. Therefore, ${\rm depth}_R\Ext^q_R(M, M)={\rm depth}_RM=d-(d-n)=n$, and hence,
for all $q\geq 0$, we have
$2n\geq d\geq d-q$, i.e., $n\geq d-n-q$. This shows that $M$ satisfies $({\bf D}_n)$.

\smallskip

Note this example also illustrates Corollary \ref{janswer} if in addition $d<2n$. Indeed, in this case we have ${\rm pd}_RM< \infty$ and ${\rm Ext}_R^n(M, M)=0$, because ${\rm pd}_RM=d-n<n$.

\medskip

(d) Let $R={\mathbb C}[\![x, y, z, u, v]\!]/(xy-uv)$. Then $R$ is a  hypersurface normal local domain with $d=4$. Consider the $R$-module  $$M={\rm coker}\,A, \quad A  = \left(\begin{array}{cc}
x  & u\\
v & y
\end{array}\right),$$ which has a (periodic) infinite minimal free resolution
$$\xymatrix@=1em{\cdots \ar[r]^{A} & R^2\ar[r]^B & R^2\ar[r]^A & R^2\ar[r] & M\ar[r]  & 0}, \quad B = ~\left(\begin{array}{cc}
y  & -u\\
-v & x
\end{array}\right).$$
First, by \cite[Remark 2.12]{Ara}, we have ${\rm Ext}^1_R(M, M)={\rm Ext}^3_R(M, M)=0$. Now, since $M\cong \Omega^2M$ (the second syzygy module of $M$), there is a short exact sequence 
$\xymatrix@=1em{0\ar[r] & M\ar[r] & R^2\ar[r] & \Omega M\ar[r] & 0}$, which induces a long exact sequence of local cohomology modules
$$\xymatrix@=1em{\cdots \ar[r] & H^i_\mathfrak{m}(M)\ar[r] & H^i_\mathfrak{m}(R^2)\ar[r] & H^i_\mathfrak{m}(\Omega M)\ar[r] & H^{i+1}_\mathfrak{m}(M)\ar[r]  & \cdots}.$$
By doing the same for the natural short exact sequence $\xymatrix@=1em{0\ar[r] & \Omega M\ar[r] & R^2\ar[r] & M\ar[r] & 0}$ and comparing the two long exact sequences, we conclude that both $M$ and $\Omega M$ are maximal Cohen-Macaulay $R$-modules.


It follows again by $M\cong \Omega^2M$ that $\Ext^2_R(M,M)\cong \Hom_R(M,M)$. Now the spectral sequence in Lemma \ref{ss2} has the following shape in its corner $$\xymatrix@=0.5em{
0\ar@{--}[rrrddd] & 0 & 0 & 0 & 0
\\
E_2^{0,2}\ar@{--}[rrdd] & E_2^{1,2} & E_2^{2,2} & E_2^{3,2} & E_2^{4,2}\\
0 & 0 & 0 & 0 & 0
\\
E_2^{0,0} & E_2^{1,0} & E_2^{2,0} & E_2^{3,0} & E_2^{4,0}}$$
with $E_2^{i,2}\cong E_2^{i,0}$ for all $i\geq 0$ and the dotted lines correspond to the convergence of this spectral sequence. Since $\depth_R\Hom_R(M,M)\geq\min\{2,\,\depth_RM\}=2$, it follows that $E_2^{0,2}=E_2^{1,2}=0$ and, by convergence (along with Lemma \ref{depthlemma}), we have $$H^i_\mathfrak{m}(\Hom_R(M,M))=E_2^{i,0}\cong H^i_\mathfrak{m}(M,M)=0 \quad \mbox{for} \quad i\leq3.$$ Therefore $\Hom_R(M,M)$ is maximal Cohen-Macaulay. This shows that $M$ satisfies $({\bf D}_1)$.

\end{example}

\section{Finiteness of other homological dimensions}\label{otherhd}

So far in this note we have dealt especially with modules of finite projective dimension. In the present section, we add further homological dimensions into our investigation and consider, more precisely, the interplay between the vanishing of Ext modules and the finiteness of the injective dimension and the Gorenstein injective dimension of a finite module. Applications to prescribed bound on projective dimension  will be given. We maintain the previous notations.

The auxiliary results below will be useful to the main theorem of this section. The injective dimension of a module $N$ over a ring $R$ is denoted $\id_RN$.

\begin{lemma}\label{cv} 
Let $R$ be a local ring. If $M$ is a maximal Cohen-Macaulay $R$-module and $N$ is a finite $R$-module with $\id_RN<\infty$, then $e_R(M,N)=0$.
\end{lemma}
\begin{proof} By Ischebeck's theorem (see \cite[2.6]{Is} or \cite[Exercise 3.1.24]{BH}), the hypothesis ${\rm id}_RN < \infty$ implies that
$e_R(M,N)={\rm depth}\,R-{\rm depth}_RM$ and, moreover, that $R$ must be Cohen-Macaulay by the classical Bass' conjecture (see \cite[Corollary 9.6.2 and Remark 9.6.4(ii)]{BH}). The result follows.
\end{proof}

The next lemma is a version of  Ischebeck's formula  in the context of the so-called \textit{Gorenstein injective dimension} introduced in \cite{EJ}, which is yet another homological invariant of a module. The Gorenstein injective dimension of an $R$-module $M$ is denoted $\gid_RM$ and generalizes the usual injective dimension in the sense that $\gid_RM\leq\id_RM$, with equality if $\id_RM<\infty$, according to \cite[Proposition 3.10]{CFH}. Another useful property is that, if $R$ is Gorenstein, then $\gid_RM<\infty$ for every $R$-module $M$; this follows from \cite[Theorem 3.2]{EJ2}.

\begin{lemma}{\rm (\cite[Theorem 2.10]{Sa})}\label{sa}
Let $M$ be a finite $R$-module with $\id_RM<\infty$ and let $N$ be a finite $R$-module with $\gid_RN<\infty$. Then, $e_R(M,N)=\depth R-\depth_RM$.
\end{lemma}

Now we recall the local duality version for finite injective dimension.

\begin{lemma}{\rm (\cite[Theorem 2.1(b)]{HZ})}\label{dualityhom}
Let $R$ be a Cohen-Macaulay local ring of dimension $d$ with canonical module $\omega_R$, and let $M,N$ be finite $R$-modules. If $\id_RN<\infty$ then, for each $j\geq0$, there is an isomorphism $H^j_\mathfrak{m}(M,N)^\vee\cong\Ext^{d-j}_{\widehat{R}}(\Hom_{\widehat{R}}(\omega_{\widehat{R}},\widehat{N}),\widehat{M})$.
\end{lemma}

\begin{definition}\rm Let $R$ be a local ring admitting a canonical module $\omega_R$ and let $M, N$ be finite $R$-modules. Given positive integers $n\leq s\leq {\rm dim}\,R$,  we say for convenience that the pair $M, N$ satisfies the $({\bf V}^{n, s})$ condition if $$\Ext^j_R(N,M)=\Ext^{j+1}_R(\Hom_R(\Omega\omega_R,N),M)=0\quad \mbox{for \,all} \quad j=n, \ldots, s.$$ 
    
\end{definition}

The theorem below is the main result of this section.

\begin{theorem}\label{prescribedpdhom}
Let $R$ be a Cohen-Macaulay local ring of dimension $d$ with canonical module $\omega_R$. If $M, N$ is a pair of finite $R$-modules with $\id_RN<\infty$ such that the $({\bf V}^{n, s})$ and $({\bf D}_{s})$ conditions hold for some positive integers $n\leq s\leq d$, then $\depth_RN>d-n$. If in addition $\gid_RM<\infty$, then $e_R(N,M)<n$.
\end{theorem}
\begin{proof}
By Lemma \ref{cv} we have $\Ext^1_R(\omega_R, N)=0$, hence there is a short exact sequence $$\xymatrix@=1em{0\ar[r] & \Hom_R(\omega_R,N)\ar[r] & \Hom_R(F,N)\ar[r] & \Hom_R(\Omega\omega_R,N)\ar[r] & 0}$$ for some finite free $R$-module $F$. This yields, for each $i\geq0$, an exact sequence $$\xymatrix@=1em{\Ext^i_R(\Hom_R(F,N),M)\ar[r] & \Ext^i_R(\Hom_R(\omega_R,N),M)\ar[r] & \Ext^{i+1}_R(\Hom_R(\Omega\omega_R,N),M).}$$ Using the $({\bf V}^{n, s})$ hypothesis, we obtain $\Ext^i_R(\Hom_R(\omega_R,N),M)=0$ for all $i=n, \ldots, s$. Now it should be noticed that $R$ can be assumed to be complete, and therefore Lemma \ref{dualityhom} ensures that $H^j_\mathfrak{m}(M,N)=0$ for all $j=d-s, \ldots, d-n$. On the other hand, by the $({\bf D}_{s})$ condition, the spectral sequence given in Lemma \ref{ss2}, $$E_2^{p,q}=H^p_\mathfrak{m}(\Ext^q_R(M,N))\Rightarrow_p H^{p+q}_\mathfrak{m}(M,N)$$ is such that $E_2^{p,q}=0$ for all $p<d-s-q$. By convergence, it follows that $H^j_\mathfrak{m}(M,N)=0$ for all $j<d-s$. Summing up, we have $H^j_\mathfrak{m}(M,N)=0$ for all  $j<d-n$. Thus, Lemma \ref{depthlemma} gives $\depth_RN>d-n$.
Finally, if $\gid_RM<\infty$ then, by Lemma \ref{sa}, we conclude that $e_R(N,M)=d-\depth_RN<n$.
\end{proof}

\begin{remark}\label{remarkhom}\rm
Concerning the vanishing hypothesis of Theorem \ref{prescribedpdhom}, it is clear from the proof above that the condition $\Ext^j_R(\Hom_R(\omega_R,N),M)=0$ for all $j=n, \ldots, s$ suffices to ensure the same conclusions.

\end{remark}

\begin{corollary}\label{prescribedpdhom-cor}
Let $R$ be a Cohen-Macaulay local ring of dimension $d$ with canonical module $\omega_R$. If $M, N$ is a pair of finite $R$-modules with $\id_RN<\infty$ such that the $({\bf V}^{1, s})$ and $({\bf D}_{s})$ conditions hold for some positive integer $s\leq d$, then $N$ is isomorphic to a direct sum of finitely many copies of $\omega_R$. If in addition $\gid_RM<\infty$, then $e_R(\omega_R,M)=0$.
\end{corollary}
\begin{proof} Applying Theorem \ref{prescribedpdhom} with $n=1$, we obtain that $N$ is maximal Cohen-Macaulay. But it is well-known that a  maximal Cohen-Macaulay $R$-module of finite injective dimension is necessarily isomorphic  to $\omega_R^{\oplus t}$ for some $t\geq 1$.
\end{proof} 

In another application, we obtain the following criterion for prescribed bound on projective dimension for certain modules of finite injective dimension.

\begin{corollary}\label{cor-e} Let $R$ be a Cohen-Macaulay local ring of dimension $d$ with canonical module $\omega_R$. Let $M, N$ be a pair of finite $R$-modules with $\id_RN<\infty$ and $\id_R\Hom_R(\omega_R,N)<\infty$. If $\Ext^j_R(N,M)=0$ for all $j=n, \ldots, s$ and the  $({\bf D}_{s})$ condition holds for the pair $M, N$, for some positive integers $n\leq s\leq d$, then $\pd_RN<n$.
\end{corollary}
\begin{proof} By virtue of $\id_R\omega_R<\infty$, we obtain ${\rm Gid}_R\omega_R<\infty$. Since in addition $\id_RN<\infty$, we are in a position to apply \cite[Corollary 2.13]{Sa} to get the equality $$\pd_R{\rm Hom}_R(\omega_R,N)=e_R(N, \omega_R).$$ On the other hand, Lemma \ref{sa} yields $e_R(N, \omega_R)<\infty$. Therefore $\pd_R\Hom_R(\omega_R,N)<\infty$, and since by hypothesis $\id_R\Hom_R(\omega_R,N)<\infty$, we derive that $R$ must be Gorenstein by a classical fact (see \cite[Corollary 4.4]{Fox}). Now we have $\Omega \omega_R=0$ and ${\rm Gid}_RM<\infty$, so that Theorem \ref{prescribedpdhom} yields $\depth_RN>d-n$. Finally, because $R$ is Gorenstein and $\id_RN<\infty$, we have $\pd_RN<\infty$ (see \cite[Exercise 3.1.25]{BH}) and then the Auslander-Buchsbaum formula gives $\pd_RN<n$.
\end{proof}

Notice that the case $n=1$ of Corollary \ref{cor-e}
gives a criterion for the freeness of $N$.

\section{Prescribed bound on projective dimension via weakly full ideals}\label{sectionfinitepd}

In this last section,  we establish consequences of Theorem \ref{prescribedpdhom}  which deal with finiteness of projective dimension over a local ring $(R, {\mathfrak m})$ via the existence of  a weakly $\mathfrak{m}$-full ideal (see Definition \ref{weakly-def} below) satisfying suitable hypotheses.

To this end, let us invoke the notion of strongly rigid module and then a series of preparatory facts. First recall that a finite $R$-module $M$ is said to be locally free on the punctured spectrum of $R$ if $M_{\mathfrak p}$ is a free $R_{\mathfrak p}$-module (e.g., $M_{\mathfrak p}=0$) for every ${\mathfrak p}\in {\rm Spec}\,R\setminus \{{\mathfrak m}\}$.  

\begin{definition}\rm Let $R$ be a local ring. A finite $R$-module $M$ is said to be {\it strongly rigid} if $\pd_RN<\infty$ whenever $N$ is a finite $R$-module with $\Tor_j^R(M,N)=0$ for some $j\geq1$.
\end{definition}

\begin{lemma}{\rm (\cite[Proposition 3.6]{CGZS})}\label{stronglyrigidfiniteid}
Let $R$ be a Cohen-Macaulay local ring of dimension $d$ with canonical module, and let $N$ be a finite $R$-module. If there exists a non-zero strongly rigid $R$-module $M$ which is locally free on the punctured spectrum of $R$ and satisfies $\Ext^i_R(M,N)=0$ for some $i\geq d+1$, then $\id_RN<\infty$. 
\end{lemma}

\begin{lemma}\label{stronglyrigidgor}{\rm (\cite[Theorem 7.3]{CGZS})}
Let $R$ be a Cohen-Macaulay local ring. If $\gid_RM<\infty$ for some non-zero strongly rigid $R$-module $M$, then $R$ is Gorenstein.
\end{lemma}

\begin{corollary}\label{stronglyrigidfinitepd}
Let $R$ be a Cohen-Macaulay local ring of dimension $d$ with canonical module $\omega_R$, and let $N$ be a finite $R$-module. Suppose there exists a non-zero strongly rigid $R$-module $M$ which is locally free on the punctured spectrum of $R$, such that
$\gid_RM<\infty$ and $\Ext^i_R(M,N)=0$ for some $i\geq d+1$. If, in addition, $\Ext^j_R(N,M)=0$ for all $j=n, \ldots, s$ and $({\bf D}_{s})$ holds for the pair $M, N$, for some positive integers $n\leq s\leq d$, then $\pd_RN<n$.
\end{corollary}
\begin{proof}
By Lemma \ref{stronglyrigidfiniteid}, we have $\id_RN<\infty$, whereas Lemma \ref{stronglyrigidgor} ensures that $R$ is Gorenstein.  Now, the Gorenstein case of Corollary \ref{cor-e} immediately yields $\pd_RN<n$.
\end{proof}

\begin{lemma}{\rm (\cite[Corollary 6.1]{CGZS})}\label{rigidtestpd}
Let $R$ be a local ring and let $M, N$ be non-zero finite $R$-modules. If $M$ is strongly rigid and $\Ext^i_R(N,M)=0$ for some $i\geq\depth R$, then $\pd_RN=e_R(N,M)<i$.
\end{lemma}

\begin{corollary}\label{cor-finite-pd}
Let $R$ be a Gorenstein local ring of dimension $d$ and let $N$ be a finite $R$-module. Suppose there exists a non-zero strongly rigid $R$-module $M$ with $\Ext^i_R(N,M)=0$ for some $i\geq d$. If, in addition, $\Ext^j_R(N,M)=0$ for all $j=n, \ldots, s$ and $({\bf D}_{s})$ holds for the pair $M, N$, for some positive integers $n\leq s\leq d$, then $\pd_RN<n$.
\end{corollary}
\begin{proof}
Lemma \ref{rigidtestpd} gives $\pd_RN<\infty$, hence $\id_RN<\infty$. The result now follows by Corollary \ref{cor-e}.
\end{proof}

\begin{definition}\label{weakly-def}\rm (\cite[Definition 2.1]{CK}) Let $(R, \mathfrak{m})$ be a local ring and let $I, J$ be ideals of $R$. We say that $I$ is {\it weakly $\mathfrak{m}$-full with respect to $J$} provided that $I:_RJ = \mathfrak{m}I:_R\mathfrak{m}J$. In case $J=R$, i.e. if $I=\mathfrak{m}I:_R\mathfrak{m}$, then $I$ is simply said to be {\it weakly $\mathfrak{m}$-full}.
\end{definition}

For example, if ${\rm depth}\,R>0$ then all integrally closed ideals of $R$ are weakly $\mathfrak{m}$-full with respect to $\mathfrak{m}^s$ for each $s\geq 0$ (see \cite[Proposition 2.4]{CK}).

\begin{lemma}{\rm (\cite[Corollary 2.14]{CK})}\label{weaklymfullpdfinite}
Let $(R, \mathfrak{m})$ be a non-regular local ring with $\depth R>0$, and let $I$ be an $\mathfrak{m}$-primary ideal of $R$ such that $I$ is weakly $\mathfrak{m}$-full with respect to $\mathfrak{m}^{\nu}$ and $I\subset\mathfrak{m}^{\nu +1}$, for some $\nu \geq0$ {\rm (}note the case $\nu = 0$ means that $I$ is  weakly $\mathfrak{m}$-full{\rm )}. In addition, let $N$ be a finite $R$-module and $t\geq 1$ be an integer. If $\Tor^R_t(R/I, N)=0$, then $\pd_RN<t$.
\end{lemma}

Note Lemma \ref{weaklymfullpdfinite} implies that $R/I$ is a strongly rigid $R$-module. This will be used in the result below.

\begin{proposition}\label{weakly-strongly}
Let $(R, \mathfrak{m})$ be a Gorenstein non-regular local ring of  dimension $d$, and let $I$ be an $\mathfrak{m}$-primary ideal of $R$ such that $I$ is weakly $\mathfrak{m}$-full with respect to $\mathfrak{m}^{\nu}$ and $I\subset\mathfrak{m}^{\nu + 1}$, for some $\nu \geq 0$. Consider positive integers $n\leq s\leq d$, and let $N$ be a finite $R$-module such that $\Ext^j_R(N, R/I)=0$ for all $j=n, \ldots, s$. Assume that the pair $R/I, N$ satisfies the  $({\bf D}_{s})$ condition. Then, $\pd_RN<n$ if any one of the following assertions is satisfied:
\begin{itemize}
\item [(i)] $\Tor^R_t(R/I, N)=0$ for some $t\geq 1$;
    \item [(ii)] $\Ext^i_R(R/I, N)=0$ for some $i\geq d+1$;
    \item [(iii)]$\Ext^i_R(N, R/I)=0$ for some $i\geq d$.
\end{itemize}
\end{proposition}
\begin{proof} In case (i) holds, Lemma \ref{weaklymfullpdfinite} yields $\pd_RN<t<\infty$, and hence $\id_RN<\infty$, which gives $\pd_RN<n$ by Corollary \ref{cor-e}. So it remains to prove the result in the other two cases. As already pointed out, $R/I$ is strongly rigid as an $R$-module. Note $\gid_RR/I<\infty$ because $R$ is Gorenstein.  Moreover, since in particular $I$ is $\mathfrak{m}$-primary, we have $(R/I)_{\mathfrak{p}}=0$ for all prime ideals $\mathfrak{p}\neq \mathfrak{m}$ and hence, trivially, $R/I$ is locally free (of rank zero) on the punctured spectrum of $R$. Now if (ii) (resp. (iii)) holds then we get $\pd_RN<n$ by Corollary \ref{stronglyrigidfinitepd} (resp. Corollary \ref{cor-finite-pd}).
\end{proof}

Clearly, criteria for the freeness of $N$ can be readily seen by taking $n=1$ in the above proposition. We close the paper with a few more comments.

\begin{remark}\rm (a) In the case that (i) holds, the result (together with Lemma \ref{weaklymfullpdfinite}) in fact yields $\pd_RN<\min\{t,n\}$.

\medskip

(b) In order to make Proposition \ref{weakly-strongly} feasible, an obstruction on the shape of $N$ must be taken into account. Precisely, $N$ cannot be of the form
$\mathfrak{m}^kN'$, where $N'$ is a non-zero finite $R$-module and $k\geq 1$ is an integer. Indeed, suppose by way of contradiction that the module $\mathfrak{m}^kN'$ fits into the hypotheses of the theorem. Then we would get  $\pd_R\mathfrak{m}^kN'<n<\infty$, which by \cite[Theorem 1.1]{LV} is equivalent to $R$ being regular. This violates our choice of $R$. 

\medskip

(c) The case $n=s=d$ of Proposition \ref{weakly-strongly} gives that the ideal $\mathfrak{m}$ must contain an $N$-regular element if $\Ext^d_R(N, R/I)=0$.

\end{remark}

\bigskip

\noindent{\bf Acknowledgements.} The second-named author was partially supported by CNPq (grants 406377/2021-9, 313357/2023-4 and 408698/2023-3). The authors thank the anonymous referee for her/his interesting suggestions, in particular a request for examples of the condition $({\bf D}_n)$, which led us to produce Example \ref{exs}.


\begin{thebibliography}{9}
\bibliographystyle{alpha}


\bibitem{Ara}
T. Araya, O. Celikbas, A. Sadeghi, R. Takahashi, {\it On the vanishing of self-extensions over Cohen-Macaulay local rings}, Proc. Amer. Math. Soc. {\bf 146} (2018), 4563--4570.

\bibitem{AB}
L. L. Avramov, R.-O. Buchweitz, {\it Support varieties and cohomology over complete intersections}, Invent. Math. {\bf 142} (2000), 285--318.


\bibitem{BH}
W. Bruns, J. Herzog, {\it Cohen–Macaulay rings}, Revised Edition, Cambridge Univ. Press, 1998.



\bibitem{CK} O. Celikbas, T. Kobayashi, {\it On a class of Burch ideals and a conjecture of Huneke and Wiegand}, Collect. Math. (2021), doi:10.1007/s13348-021-00315-8.

\bibitem{Cr00} L. W. Christensen, {\it Gorenstein dimensions}, Lecture Notes in Math. {\bf 1747}, Springer-Verlag, Berlin, 2000.  

\bibitem{CFH} L. W. Christensen, H.-B. Foxby, H. Holm, {\it Beyond totally reflexive modules and back. A survey on Gorenstein dimensions}, in: {\it Commutative Algebra, Noetherian and Non-Noetherian
Perspectives}, Springer, New York, 2011.

\bibitem{Dao} H. Dao, {\it Some observations on local and projective hypersurfaces}, Math. Res. Lett. {\bf 15} (2008), 207--219.

\bibitem{EJ}
E. E. Enochs, O. M. G. Jenda, {\it Gorenstein injective and projective modules}, Math. Z. {\bf 220} (1995), 611--633.

\bibitem{EJ2}
E. E. Enochs, O. M. G. Jenda, {\it Gorenstein balance of Hom and tensor}, Tsukuba J. Math. {\bf 19} (1995), 1--13.


\bibitem{Fox} H.-B. Foxby, {\it Isomorphisms between complexes with applications to the homological theory of
modules}, Math. Scand. {\bf 40} (1977), 5--19.

\bibitem{FJMS}
T. H. Freitas, V. H. Jorge-P\'erez, C. B. Miranda-Neto, P. Schenzel, \emph{Generalized local duality, canonical modules, and prescribed bound on projective dimension}, J. Pure Appl. Algebra {\bf 227} (2023), 107188.

\bibitem{H} J. Herzog, {\it Komplexe, Aufl\"osungen und Dualit\"at in
	der lokalen Algebra}, Habilitationsschrift, Germany, Universit\"at Regensburg, 1970.

\bibitem{HZ}
J. Herzog, N. Zamani, {\it Duality and vanishing of generalized local cohomology}, Arch. Math. {\bf 81} (2003), 512--519.



\bibitem{Is} F. Ischebeck, {\it Eine Dualit\"at zwischen den Funktoren Ext und Tor}, J. Algebra {\bf 11} (1969), 510--531.


\bibitem{J}
D. A. Jorgensen, {\it Finite projective dimension and the vanishing of $\Ext_R(M,M)$}, Comm. Algebra {\bf 36} (2008), 4461--4471.





\bibitem{LV} G. Levin, W. V. Vasconcelos, {\it Homological dimensions and Macaulay rings}, Pacific J. Math. {\bf 25} (1968), 315--323.

\bibitem{M}
H. Matsumura, {\it Commutative ring theory}, Translated from the Japanese by M. Reid, Second edition, Cambridge Studies in Advanced Mathematics \textbf{8}, Cambridge University Press, Cambridge, 1989.

\bibitem{Sa}
R. Sazeedeh, {\it Gorenstein injective modules and a generalization of Ischebeck formula}, J. Algebra Appl. {\bf 12} (2013), 1250197.


\bibitem{Soto} A. J. Soto Levins, {\it A rigidity theorem for Ext}, J. Commut. Algebra {\bf 16} (2024), 115--122.

\bibitem{S}
N. Suzuki, {\it On the generalized local cohomology and its duality}, J. Math. Kyoto Univ. {\bf 18} (1978), 71--85.



\bibitem{Y}
K. Yoshida, (1998), {\it Tensor products of perfect modules and maximal surjective Buchsbaum modules}, J. Pure Appl. Algebra {\bf 123} (1998), 313--323.

\bibitem{CGZS}
 M. R. Zargar, O. Celikbas, M. Gheibi, A. Sadeghi, {\it Homological dimensions of rigid modules}, Kyoto J. Math. {\bf 58} (2018), 639--669.


\end{thebibliography}
\end{document}